\documentclass[12pt]{article}
\usepackage{diagbox}
\usepackage{mathtools}
\usepackage{bbm}
\usepackage{amsmath}
\usepackage[hidelinks]{hyperref}
\usepackage{comment}
\allowdisplaybreaks

\renewcommand{\emph}[1]{\textit{#1}}
\usepackage{enumerate,amsmath,amsthm,latexsym,amssymb}
\usepackage{color}\usepackage{graphicx}

\definecolor{brown}{cmyk}{0, 0.72, 1, 0.45}
\definecolor{grey}{gray}{0.5}

\newcommand{\old}[1]{}

\newcounter{rot}

\newcommand{\card}[1]{\left|#1\right|}

\newcommand{\ignore}[1]{}

\def\cC{{\mathcal C}}

\newcommand{\set}[1]{\left\{#1\right\}}

\def\ii_(#1,#2){i_{#1}^{#2}}

\def\a{\alpha}
\def\b{\beta}
\def\d{\delta}
\def\D{\Delta}
\def\e{\varepsilon}
\def\f{\phi}
\def\F{\Phi}
\def\g{\gamma}

\def\z{\zeta}

\def\th{\theta}

\def\m{\mu}

\def\r{\rho}

\def\s{\sigma}
\def\t{\tau}

\def\1{{\bf 1}}
\def\0{{\bf 0}}

\newcommand{\rdup}[1]{\left\lceil #1 \right\rceil}

\newcommand{\brac}[1]{\left( #1 \right)}

\def\E{{\bf E}}

\renewcommand{\Pr}{\operatorname{\bf Pr}}
\newcommand\bfrac[2]{\left(\frac{#1}{#2}\right)}

\newtheorem{theorem}{Theorem}[section]

\newtheorem{lemma}[theorem]{Lemma}

\newtheorem{remthm}[theorem]{Remark}

\newcounter{thmtemp}

\newcommand{\nospace}[1]{}

\def\path{\operatorname{PATH}}

\newcommand{\beq}[2]{\begin{equation}\label{#1}#2\end{equation}}

\def\leb{\leq_{\mathcal{O}}}

\title{The game chromatic number of a random hypergraph}
\author{Debsoumya Chakraborti, Alan Frieze\thanks{Research supported in part by NSF grant DMS1362785
}, Mihir Hasabnis \\Department of Mathematical Sciences\\Carnegie Mellon University\\Pittsburgh PA 15213} 
\begin{document}
\maketitle

\begin{abstract}
We consider the following game, played on a $k$-uniform hypergraph $H$. There are $q$ colors available and two players take it in turns to color vertices. A partial coloring is proper if no edge is mono-chromatic. One player, A, wishes to color all the vertices and the other player, B,  wishes to prevent this. The {\em game chromatic number} $\chi_g(H)$ is the minimum number of colors for which A has a winning strategy. We consider this in the context of a random $k$-uniform hypergraph and prove upper and lower bounds that hold w.h.p.
\end{abstract}
\section{Introduction}
Let $G=(V,E)$ be a graph and let $q$ be a positive integer. Consider the following game in which two players A(lice) and B(ob) take turns in coloring the vertices of $G$ with $q$ colors. Each move consists of choosing an uncolored vertex of the graph and assigning to it a color from $\{1, \ldots, k\}$ so that the resulting coloring is {\em proper}, i.e., adjacent vertices get different colors. A wins if all the vertices of $G$ are eventually colored. B wins if at some point in the game the current partial coloring cannot be extended to a complete coloring of $G$, i.e., there is an uncolored vertex such that each of the $q$ colors appears at least once in its neighborhood. We assume that A goes first (our results will not be sensitive to this choice). The {\em game chromatic number} $\chi_g(G)$ is the least integer $q$ for which A has a winning
strategy.

This parameter is well defined, since it is easy to see that A always wins if the number of colors is larger than the maximum degree of $G$.  Clearly, $\chi_g(G)$ is at least as large as the ordinary chromatic number $\chi(G)$, but it can be considerably more. The game was first considered by Brams about 25 years ago in the context of coloring planar graphs and was described in Martin Gardner's column \cite{Ga} in Scientific American in 1981. The game remained unnoticed by the graph-theoretic community until Bodlaender \cite{Bod} re-invented it. For a survey see Bartnicki, Grytczuk, Kierstead and Zhu \cite{BGKZ}.

The papers by Bohman, Frieze and Sudakov \cite{BFS}, Frieze, Haber and Lavrov \cite{FHL} and by Keusch and Steger \cite{KS} discuss the game chromatic number of random graphs. In this paper we discuss the game chromatic number of random hypergraphs. Given a hypergraph $H=(V,E)$ we can consider basically the same game. Here A, B color the vertices of $H$ consecutively and a coloring is proper if there is no $e\in E$ such that all vertices in $e$ have the same color. This problem has hardly been studied, even in a deterministic setting, but we feel it is of interest to extend the results of \cite{BFS}, \cite{FHL} and \cite{KS} to this setting. 

In this paper we will restrict our attention to the random $k$-uniform hypergraph $H_{n,p;k}$ where each of the $\binom{n}{k}$ edges appear independently with probability $p=\frac{d}{n^{k-1}}$ where $d$ is a large constant. Now Krivelevich and Sudakov \cite{KSchrom} have shown that 
\beq{chik3}{
\chi(H_{n,p;k})\approx \bfrac{d}{k(k-2)!\log d}^{1/(k-1)}
}
Here we use the notation $A_n\approx B_n$ for sequences $A_n,B_n,n\geq1$ to mean that $A_n=(1+o(1))B_n$ as $n\to\infty$.

Our first theorem shows that w.h.p. the game chromatic number $\chi_g$ is significantly larger than the chromatic number.
\begin{theorem}\label{th1}
There exists a constant $\e>0$ such that w.h.p.,
$$\chi_g(H_{n,p;k})\geq (1+\e)\bfrac{d}{k(k-2)!\log d}^{1/(k-1)},\qquad\text{ if $d$ is sufficiently large}.$$
\end{theorem}

We also prove an upper bound in the case $k=3$ that is somewhat far from that implied by \eqref{chik3}. 
\begin{theorem}\label{th2}
Let $\d>0$ be arbitrary. Then w.h.p.,
$$\chi_g(H_{n,p;3})\leq d^{2/3+\d},\qquad\text{ if $d$ is sufficiently large}.$$
\end{theorem}
It is natural to state the following:\\
{\bf Conjecture:} W.h.p. $\chi_g(H_{n,p;k})=O\bfrac{d}{k!\log d}^{1/(k-1)}$.

We often refer to the following Chernoff-type bounds for the tails of binomial
distributions (see, e.g., \cite{AS} or \cite{JLR}). Let $X=\sum_{i=1}^n X_i$ be a sum of
independent indicator random variables such that $\Pr(X_i=1)=p_i$ and let
$p=(p_1+\cdots+p_n)/n$. Then
\begin{align}
\Pr(X\leq (1-\e)np)&\leq e^{-\e^2np/2},\label{chl}\\
\Pr(X\geq (1+\e)np)&\leq e^{-\e^2np/3},\qquad\e\leq 1,\label{chu1}\\
\Pr(X\geq \m np)&\leq (e/\m)^{\m np}.\label{chu2}
\end{align}
\section{Lower Bound}
Let
\[
D=\bfrac{d}{k!\log d}^{1/(k-1)}
\]
and suppose that there are $q=\a D$ colors available.

Bob's strategy is to choose the same color as Alice, $i$ say, but assign it randomly to one of the set of available vertices for color $i$. 

{\bf Notation:} Let $C_i=C_i(t)$ be the set of vertices that have been colored $i$ after $t$ rounds. Let $S_i=S_i(t)$ be the set of vertices that were colored by B. Let $C=C(t)=\cup_{i=0}^q C_i$ denote the partial coloring of the vertex set. 

\begin{lemma}\label{lem1}
Suppose we run this process for $t=\theta n,\th<1/2$ many rounds and that $|C_i(t)|=2\beta n/D$. We show that if $d$ is sufficiently large and 
\beq{formula}{
2(2\b+\g)^k-(2\b)^k>\frac{2(\b+\g)}{k-1}
}
then with probability $1-o(1/n)$ there exists no set $T$ such that (i) $C_i\cap T=\emptyset$, (ii) $C_i\cup T$ is independent and $|T|=\g n/D$.
\end{lemma}
The reader can easily check that \eqref{formula} is satisfied for $k\geq 3$ and
$$\b=\frac{1-2\e}{2(k-1)^{1/(k-1)}},\quad \g=\frac{\e}{(k-1)^{1/(k-1)}}$$
when $\e>0$ is sufficnetly small. The proof of the lemma is deferred to Section \ref{seclem1}. 

If the event $\set{\exists i: C_i,T}$ does not occur then
because no color class has size greater than $(2\b+\g)n/D$
the number $\ell$ of colors $i$ for which $|S_i|\geq \b n/D$ by this time satisfies
$$\frac{\ell(2\b+\g)}{D}+\frac{2(q-\ell)\b}{ D}\geq 2\th\text{ or }\frac{\ell\g+2q\b}{D}\geq 2\th.$$
We choose $\a=(1+\e)(k-1)^{1/(k-1)}$ and $\th=\a(2\b+\g)/2<1/2$. Since $q\geq \ell$,
this implies that
$$\frac{q}{D}\geq \frac{2\th}{2\b+\g}=\a.$$
This completes the proof of Theorem \ref{th1}, after replacing $(1+\e)$ by $(1+\e)^{1/(k-1)}$ for aesthetic purposes.
\subsection{Proof of Lemma \ref{lem1}}\label{seclem1}
For expressions $X,Y$ we sometimes use the notation $X\leb Y$ in place of $X=O(Y)$ when the
bracketing is ``ugly''. 

Now, (explanations for \eqref{f0}, \eqref{f1}, \eqref{sumgone} and \eqref{9} below), if $d$ is sufficiently large then
\begin{align}
  &\Pr(\exists i,C_i,T)\nonumber\\
& \leb q \binom{n}{\frac{\beta n}{D}} \binom{n}{\frac{\gamma n}{D}}  \sum_{\left|S\right| = \beta n/D} P(S_i = S) (1-p)^{\frac{(2 \beta + \gamma)^k n^k}{k!D^k}} \label{f0}\\
&\leq  q  \binom{n}{\frac{\beta n}{D}} \binom{n}{\frac{\gamma n}{D}}\sum_{\left|S\right| = \beta n/D} \bfrac{\beta n}{D}! (1-p)^{\frac{(2 \beta + \gamma)^k n^k}{k!D^k}} \prod_{j=1}^{\b n/D}\frac{7}{(1-2\th)n(1-p)^{\brac{\frac{\b n}{D} -j+1}\frac{2(2j-1)^{k-2}}{(k-2)!}}}\label{f1}\\
& \leb q  \binom{n}{\frac{\beta n}{D}}^2 \binom{n}{\frac{\gamma n}{D}}\bfrac{\beta n}{D}! \frac{7^{\frac{\beta n}{D}}(1-p)^{\frac{(2 \beta + \gamma)^k n^k}{k!D^k}} }{((1-2 \theta)n)^{\beta n/D} (1-p)^{\sum_{j=1}^{\b n/D}\brac{\frac{\b n}{D} -j+1}\frac{2(2j-1)^{k-2}}{(k-2)!}}}\nonumber\\
& \leb q \binom{n}{\frac{\beta n}{D}}^2 \binom{n}{\frac{\gamma n}{D}}\bfrac{\beta n}{D}! \frac{7^{\frac{\beta n}{D}}(1-p)^{\frac{(2 \beta + \gamma)^k n^k}{k!D^k}} }{((1-2 \theta)n)^{\beta n/D} (1-p)^{\frac{(2\b n)^k }{ 2k! D^k} }} \label{sumgone}\\
& \leb q \binom{n}{\frac{\beta n}{D}}^2 \binom{n}{\frac{\gamma n}{D}}\bfrac{\beta n}{D}!(1-p)^{\brac{\frac{2n^k((2 \beta + \gamma)^k)- (2\beta)^k}{2k!D^k}}} \frac{7^{\frac{\beta n}{D}}}{((1-2 \theta)n)^{\frac{\beta n}{D}}}\nonumber\\
&\leb qn^{1/2}\brac{\frac{(eD)^{2\beta+\gamma}}{\beta^{2\beta}\gamma^\gamma}\bfrac{7\b}{De(1-2\th)}^\b \exp\set{-\frac{d}{D^{k-1}}\brac{\frac{2(2\beta +\gamma)^k -(2\b)^k}{2k!}}}}^{\frac{n}{D}}\nonumber\\
&\leq \brac{cD^{\beta+\gamma} \exp\set{-\brac{\frac{2(2\b+\g)^k-(2\b)^k}{2}\log d}}}^{\frac{n}{D}},\label{9}\\
\noalign{(\text{where }$c=c(\theta, \beta, \gamma)=O(1)$)}\nonumber\\
&=o(1/n),\nonumber
\end{align}
assuming \eqref{formula}.

{\bf Justifying \eqref{sumgone}:}
\begin{align*}
&\sum_{j=1}^{\b n/D}\brac{\frac{\b n}{D} -j+1}\frac{2(2j-1)^{k-2}}{(k-2)!}\\
&\approx \int_{x=1}^{\b n/D} \brac{\frac{\b n}{D} -x+1}\frac{2(2x-1)^{k-2}}{(k-2)!}dx\\
&=\frac{1}{(k-2)!}\int_{x=1}^{\b n/D}\brac{\brac{\frac{2\b n}{D}+1}(2x-1)^{k-2}-(2x-1)^{k-1}}dx\\
&=\brac{\frac{2\b n}{D}+1}\brac{\frac{(2\b n/D-1)^{k-1}-1}{2(k-1)!}}-\frac{(2\b n/D-1)^{k}-1}{2k(k-2)!}\\
&\leq \frac{(2\b n/D)^k}{2k!}.
\end{align*}

{\bf Justifying \eqref{9}:} We used the asymptotic formula for summation of $k$-th power of first $n$ natural numbers, i.e. $\sum_{i=1}^n i^k \approx \frac{n^{k+1}}{k+1}$

{\bf Justifying \eqref{f0}:}
There are $q$ choices for color $i$. Then we take the union bound over all $\binom{n}{\b n/D}\binom{n}{\g n/D}$ possible choices of $C_i\setminus S_i$ and $T$. In some sense we are allowing Alice to simultaneously choose all possible sets of size $\b n/D$ for $C_i\setminus S_i$. The union bound shows that w.h.p.\ all choices fail. We do not sum over orderings of $C_i\setminus S_i$. We instead compute an upper bound on $\Pr(S_i=S)$ that holds regardless of the order in which Alice plays. We consider the situation after $\theta n$ rounds. That is, we think of the following random process: pick a $k$-uniform hypergraph $H \sim H(n,p;k)$, let Alice play the coloring game on $H$ with $q$ colors against a player who randomly chooses an available vertex to be colored by the same color as Alice. Stop after $\theta n$ moves. At this point Alice played with color $i$ and there are $\b n/D$ vertices that were colored $i$ by Alice and the same number that were colored $i$ by Bob. We bound the probability that at this point there are $\g n/D$ vertices that form an independent set with the color class $C_i$. We take a union bound over all the possible sets for Alice's vertices and for the vertices in $T$. The probability of Bob choosing a certain set is computed next.

{\bf Justifying \eqref{f1}:} Consider a sequence of random variables $X_1=N=(1- 2\theta)n,X_j= Bin(X_{j-1}, p_j)$ where $p_j=(1-p)^{\frac{2(2j-1)^{k-2}}{(k-2)!}}, 2\leq j\leq t$. $X_j$ is a lower bound for the number of vertices that Bob can color $i$ and $p_j$ is a lower bound on the probability that a vertex $v$ that was $i$-available at step $j-1$ is also $i$-available after step $j$. The probability that a vertex $v$ was $i$-available at time $j-1$ and is still $i$-available now is at least $(1-p)^{\binom{2j-2}{k-3}}(1-p)^{2\binom{2j-2}{k-2}} \ge p_j$. Our estimate for $p_j$ arises as follows: there are at most $2(j-1)$ vertices $x_1, \dots, x_{2(j-1)}$ of color $i$ and each of the vertices $z=y,y'$ colored $i$ in round $j$ yield possible edges $\set{v,z,x_i,\ldots}$ that could remove $v$ from $C_i$. There are also the edges $\set{v,y,y',x_i,\ldots}$ to account for. 

We need to estimate $\E(Y_t)$ where $Y_t=1/(X_1 X_2...X_t)$. $1/X_j$ is an upper bound for the probability that Bob chooses a particular vertex at step $j$ and then $Y_{\beta n/D}$ is an upper bound on the probability that Bob's sequence of choices is $x_1, x_2,....x_{\beta n/D}$ where $S= \{x_1, x_2,...x_{\beta n/D}\}$.  The following lemma is proven in \cite{FHL}:
\begin{lemma}\label{Bins}
If $B = Bin(v,\r)$, then $\E(\prod_{i=1}^q \frac{1}{B+i-1}) \leq \frac{7}{\r^q} \prod_{i=1}^q \frac{1}{v + i}$. 
\end{lemma}
Using Lemma \ref{Bins} we see that 
\begin{align*}
\Pr(S_i=S)&\leq \E\bfrac{1}{X_1 X_2...X_t}\\
&\leq\E\bfrac{7}{X_1\cdots X_{t-1}(X_{t-1}+1)p_t}\\
&\leq\E\bfrac{7^2}{X_1\cdots X_{t-2}(X_{t-2}+1)(X_{t-2}+2)p_{t-1}^2p_t}\\
\vdots\\
&\leq \prod_{j=1}^t\frac{7}{(N+j-1)(1-p)^{(t-j+1)\frac{2(2j-1)^{k-2}}{(k-2)!}}}\\
&\leq \prod_{j=1}^t\frac{7}{N(1-p)^{(t-j+1)\frac{2(2j-1)^{k-2}}{(k-2)!}}}.
\end{align*}
This completes the justification of \eqref{f1}.

\section{Upper Bound}
\subsection{Simple density properties}
For $S\subseteq [n]$ and $k=2,3$ we let $e_3(S)=|\set{f\in E:\;f\subseteq S}$ and $$e_2(S)=\set{\set{x,y}\subseteq S:\;\exists y\notin S\text{ such that }\set{x,y,z}\in E}.$$
\begin{lemma}\label{lemma1} 
If $\theta > 1$ and 
$$\bfrac{\sigma e d}{2 \theta}^{\theta} \leq \frac{\sigma}{2 e}$$ 
then w.h.p there does not exist $S\subset$ $[n],|S|\leq \sigma n$ such that $e_2(S) \geq \theta|S|$.
\end{lemma}
\begin{proof}
\begin{align*}
\Pr(\exists S, |S| \leq \sigma n, e_{2}(S) \geq \theta  |S|) & \leq \sum_{s=2 \theta}^{\sigma n} {n \choose s} {{s \choose 2} \choose \theta s} \Big(1- \Big(1-\frac{d}{n^2}\Big)^{n-2}\Big)^{\theta s}\\
& \leq \sum_{s=2 \theta}^{\sigma n}\left(\frac{n e}{s}\right)^s \left(\frac{e s}{2 \theta}\right)^{\theta s}\left(1-  \left(1 - \frac{d}{n}\right)\right)^{\theta s}\\ 
& \leq \sum_{s=2 \theta}^{\sigma n} \left(\left(\frac{n e}{s}\right) \left(\frac{e s}{2 \theta}\frac{d}{n}\right)^{\theta}\right)^{s}\\
&=\sum_{s=2\th}^{\s n}\brac{e\bfrac{s}{n}^{\th-1}\bfrac{ed}{2\th}^{\th}}^s\\
&=O\bfrac{d^\th}{n^{\th-1}}= o(1)
\end{align*}
\end{proof}

\begin{lemma} \label{lemma2} 
If $\theta > 1/2$ and 
$$\bfrac{\sigma^2 e d}{6\theta}^{\theta} \leq \frac{\sigma}{2 e}$$ 
then w.h.p. there does not exists $S\subseteq[n],|S| \leq \sigma n$ such that $e_3(S) \geq \theta |S|$.
\end{lemma}
\begin{proof}
\begin{align*}
\Pr(\exists S,|S|\leq \sigma n, e_{3}(S)\geq\theta |S|) 
&\leq \sum_{s = \sqrt{\theta}}^{\sigma n} {n \choose s} {{s \choose 3} \choose \theta s} \left(\frac{d}{n^2}\right)^{\theta s}\\ 
&\leq \sum_{s= \sqrt{\theta}}^{\sigma n} \left( \frac{ne}{s} \cdot   \left(\frac{e s^2 d}{6 \theta n^2}   \right)^{\theta}  \right)^s\\
&= \sum_{s = \sqrt{\theta}}^{\sigma n} \left( e \left(\frac{s}{n}\right)^{2\theta-1} \cdot \left(\frac{ed}{6\theta}\right)^{\theta}     \right)^s   \\
&= O\left(\frac{d^\theta}{n^{2\theta-1}} \right)\\
&= o(1)
\end{align*}
\end{proof}

For $S\subseteq [n]$ and $k=1,2$ and vertex $v$, we let $d_{S,k}(v)$ denote the number of edges $\set{v,x,y}$ such that $|\set{x,y}\cap S|=k$.

\begin{lemma}\label{lemma3}
Let $\sigma$ and $\theta$ be as in Lemma \ref{lemma1}. If $(\Delta - 2 \theta) \tau > 1$ and 
$$\left(\frac{\sigma e d}{\left(\Delta - 2 \theta \right)\tau}\right)^{\left(\Delta - 2  \theta\right)\tau} \leq \frac{\sigma}{4 e}$$
then w.h.p there do not exist $S \supseteq T$ such that $|S| = s \leq \sigma n, |T| \geq \tau s$ and $d_{S,1}(v) \geq \Delta,\forall v \in T$.
\end{lemma}
\begin{proof} 
In the light of Lemma \ref{lemma1}, the assumptions imply that w.h.p.\ $|e_1(T:S\setminus T)|\geq (\D-2\th)\t s$.
In which case,
\begin{align}
&\Pr(\exists S\supseteq T,\,|S|\leq \s n,\,|T|\geq \t s: |e_1(T:S\setminus T)|\geq (\D-2\th)\t s)\nonumber\\
&\leq\sum_{s=2\th}^{\s n}\sum_{t=\t s}^s\binom{n}{s}\binom{s}{t}\binom{st}{(\D-2\th)\t s} \Big(1-\Big(1-\frac{d}{n^2}\Big)^{n-2}\Big)^{(\D-2\th)\t s}\label{5}\\
&\leq \sum_{s=2\th}^{\s n}\sum_{t=\t s}^s\bfrac{ne}{s}^s\cdot 2^s\cdot
\bfrac{eds}{(\D-2\th)\t n}^{(\D-2\th)\t s}\nonumber\\
&=\sum_{s=2\th}^{\s n}\sum_{t=\t s}^s\brac{\frac{2ne}{s}\cdot
\bfrac{eds}{(\D-2\th)\t n}^{(\D-2\th)\t }}^s\nonumber\\
&=\sum_{s=2\th}^{\s n}\sum_{t=\t s}^s\brac{2e\bfrac{s}{n}^{(\D-2\th)\t -1}\cdot
\bfrac{ed}{(\D-2\th)\t}^{(\D-2\th)\t }}^s\label{4}\\
&=O\bfrac{d^{(\D-2\th)\t}}{n^{(\D-2\th)\t-1}}=o(1).\nonumber
\end{align}
\end{proof} 

\begin{lemma}\label{lemma4} 
Let $\sigma$ and $\theta$ be as in Lemma \ref{lemma2} and $(\Delta-3\th)\tau > 1$ and
\begin{align*}
\bfrac{\sigma^2ed}{2(\Delta-3\th)\tau }^{(\Delta-3\th)\tau} \leq \frac{\s}{4e}
\end{align*}
then w.h.p there do not exist $S \supseteq T$ such that $|S| = s \leq \sigma n, |T| \geq \tau s$ and $d_{S,2}(v) \geq \Delta,\forall  v \in  T$   
\end{lemma}
\begin{proof} 
In the light of Lemma \ref{lemma2}, the assumptions imply that w.h.p. $|e_2(T:S\setminus T)|\geq (\D-3\th)\t s$. In which case,
\begin{align*}
&\Pr(\exists S\supseteq T,\,|S|\leq \s n,\,|T|\geq \t s: |e_1(T:S\setminus T)|\geq (\D-3\th)\t s)\nonumber\\
&\leq\sum_{s=2\th}^{\s n}\sum_{t=\t s}^s\binom{n}{s}\binom{s}{t}\binom{s^2t/2}{(\D-3\th)\t s}
\bfrac{d}{n^2}^{(\D-3\th)\t s}\\
&\leq \sum_{s=2\th}^{\s n}\sum_{t=\t s}^s\bfrac{ne}{s}^s\cdot 2^s\cdot
\bfrac{eds^2}{2(\D-3\th)\t n^2}^{(\D-3\th)\t s}\\
&=\sum_{s=2\th}^{\s n}\sum_{t=\t s}^s\brac{\frac{2ne}{s}\cdot
\bfrac{eds^2}{2(\D-3\th)\t n^2}^{(\D-3\th)\t }}^s\\
&=\sum_{s=2\th}^{\s n}\sum_{t=\t s}^s\brac{2e\bfrac{s}{n}^{2(\D-3\th)\t -1}\cdot
\bfrac{ed}{2(\D-3\th)\t}^{(\D-3\th)\t }}^s\label{4}\\
&=O\bfrac{d^{(\D-3\th)\t}}{n^{(\D-3\th)\t-1}}=o(1).
\end{align*}
\end{proof}
Now let
$$q=d^{2/3 + \d} \text{ and } \b= \frac{q}{3}  \text{ and } \g= \frac{14\log d}{q},$$
for some small absolute constant $\d>0$.

We will now argue that w.h.p.\ A can win the game if $q$ colors are available.

A's initial strategy will be the same as that described in \cite{BFS}. Let $\cC=(C_1,C_2,\ldots,C_q)$ be a collection of pairwise disjoint subsets of $[n]$, i.e. a (partial) coloring. Let $\bigcup\cC$ denote $\bigcup_{i=1}^q C_i$. For a vertex $v$ let  
\[A(v,\cC)= \set{i\in [q]:\;v\mbox{ is not in an edge} \; \set{v,x,y} \; \mbox{such that} \; x,y \in C_i}, \]
and set
$$a(v,\cC)=|A(v,\cC)|.$$
Note that \( A(v, \cC) \) is the set of colors that are available at vertex \(v\) when
the partial coloring is given by the sets in \( \cC \) and \( v \not\in \bigcup \cC\).
A's initial strategy can now be easily defined. Given the current color
classes $\cC$, A chooses an uncolored vertex $v$ with the smallest
value of $a(v,\cC)$ and colors it by any available color.

As the game evolves, we let \(u\) denote the
number of uncolored vertices in the graph.  So, we think of \(u\) as running ``backward''
from \( n \) to \(0\).

We show next that w.h.p.\ every $q$-coloring (proper or improper) of the full vertex set
has the property that there are at most $\g n$ vertices with less than $\b/2$ available colors. For this we need the following lemma.
\begin{lemma}\label{deb}
$p=d/n^2$ and let $x_0=(2/(-\log(1-p)))^{1/2}\approx (2/d)^{1/2}n$ and 
$$f(C) = \sum_{i = 1}^q (1-p)^{c_i^2/2}\text{ where }\sum_{i=1}^q c_i = n.$$ 
Then we have that for $n$ sufficiently large, 
$$f(C) \geq \begin{cases}q(1-p)^{n^2/2q^2}&qx_0\leq n\\ q(1-p)^{x_0^2/2}&qx_0>n\end{cases}.$$ 
\end{lemma}
\begin{proof}
We have that the function $\f(x)=(1-p)^{x^2/2}$ is convex in the interval $[x_0,\infty]$. It follows from convexity that if $I=\set{i:c_i\leq x_0}$ then
\beq{x0}{
f(C)\geq |I|(1-p)^{x_0^2/2}+(q-|I|)(1-p)^{(n-|I|x_0)^2/2(q-|I|)^2}.
}
Suppose now that $x_0\leq n/q$. Then $\frac{n-|I|x_0}{q-|I|}\geq \frac{n}{q}$. And then convexity and \eqref{x0} implies
\beq{x1}{
f(C)\geq q\brac{\frac{|I|}{q}\f(x_0)+\frac{q-|I|}{q}\f\bfrac{n-|I|x_0}{q-|I|}}\geq q\f\bfrac{n}q.
}
If  $x_0> n/q$ then $\frac{n-|I|x_0}{q-|I|}< \frac{n}{q}$ and then \eqref{x0} implies that 
$$f(C)\geq q(1-p)^{x_0^2/2}.$$
\end{proof}
Let
$$B(\cC)=\set{v:\;a(v,\cC)< \b/2}.$$
\begin{lemma}\label{lem5}
W.h.p., for all collections $\cC$,
$$|B(\cC)|\leq \g n.$$
\end{lemma}

\begin{proof}
We first note that if $|S|=\g n$ then w.h.p. $e_2(S)\leq 4\g^2dn$. This follows from Lemma \ref{lemma1} with $\s=\g$ and $\th=4\g d$. It follows that for any $\e>0$ that there is a set $S_1'\subseteq S$ of size at least $(1-\e)\g n$ such that if $v\in S_1'$ then $d_{S,1}(v)\leq 8\e^{-1}\g d$. Furthermore, Lemma \ref{lemma2} with $\s=\g$ and $\th=3$ implies that w.h.p. $e_3(S)\leq 3\g n$. Therefore there is a set $S_1''\subseteq S$ of size at least $(1-\e)\g n$ such that if $v\in S_1''$ then $d_{S,2}(v)\leq \e^{-1}$. Let $S_1=S_1'\cup S_1''$.

Fix $\cC$ and suppose that $|B(\cC)|\geq \g n$. Choose $S\subseteq B(\cC)$ and let $S_1$ be as defined above. For $v\in S_1$ let
$$b(v,\cC)=\card{\set{i\in [q]:\;v\mbox{ is not in an edge }\set{v,x,y} 
\mbox{ such that } \; x,y \in C_i \setminus S}}.$$ 

Thus $a(v,\cC)\geq b(v,\cC)-8\e^{-1}\g d-\e^{-1}$. $b(v,\cC)$ is the sum of independent indicator variables $X_i$, where $X_i=1$ if $v$ is not in a hyperedge $(v,x,y)$ such that $x,y \in C_i\setminus S$ in $G_{n,p}$. Then $\Pr(X_i=1)\geq (1-p)^{\binom{|C_i|}{2}}$ and since $(1-p)^t$ is a convex function of $t$ and using the Lemma \ref{deb} we get
$$\E(b(v,\cC))\geq\sum_{i=1}^q(1-p)^{\binom{|C_i|}{2}}\geq\b.$$
It follows from the Chernoff bound \eqref{chl} that
$$\Pr(b(v,\cC)\leq 0.51\b)\leq e^{-\b/9}.$$
Now, when $\cC$ is fixed, the events $\set{b(v,\cC)\leq 0.51\b},v\in S_1$ are independent. Thus, because $a(v,\cC)\leq \b/2$ implies that $b(v,\cC)\leq 0.51\b$ we have 
\begin{align}
&\Pr(\exists\cC:|B(\cC)|\geq \g n)\nonumber\\
&\leq q^n\binom{n}{(1-\e)\g n}e^{-(1-\e)\g\b n/9}\nonumber\\
&\leq q^n\brac{\frac{e}{(1-\e)\g}\exp \set{-\frac{\b}{9}}}^{(1-\e)\g n}\label{1}\\
&=\exp\set{n\brac{\log q+\frac{14(1-\e)\log d}{q}\brac{\log\bfrac{e}{1-\e}-\frac{q}{27}}}}\nonumber\\
&=o(1),\nonumber
\end{align}
for large $d$ and small enough $\e$.
\end{proof}

Let $u_0$ to be the last time for which A colors a vertex with at least
$\b/2$ available colors, i.e.,
$$u_0= \min\set{u:\;a(v,\cC_u)\geq \b/2, \mbox{ for all }v \not\in \bigcup \cC_u}, $$
where $\cC_u$ denotes the collection of color classes when \(u\) vertices remain uncolored.

If $u_0$ does not exist then A will win. It follows from Lemma \ref{lem5} that w.h.p.\ $u_0\leq 2\g n$ and that at time $u_0$, every vertex still has at least $\b/2$ available colors. Indeed, consider the final coloring $\cC^*$ in the game that would be achieved if A follows her current strategy, even if she has to improperly color an edge. Let $U=\set{v\notin \cC_{u_0}:a(v,\cC^*)< \b/2}$. Now we can assume that $|U|\leq \g n$. Because the number of colors available to a vertex decreases as vertices get colored, from $u_0$ onward, every vertex colored by $A$ is in $U$. Therefore $u_0\leq 2\g n$. Next let $G_U=(U,F)$ be the graph with vertex set $U$ and edges $F$ where $\set{x,y}\in F$ if there exists $z$ such that $\set{x,y,z}\in E$.
 
Now let $u_1$ be the first time that there are at most $2\g n$ uncolored vertices and $a(v,\cC_u)\geq \b/2, \mbox{ for all }v \not\in \bigcup \cC_u$. By the above, w.h.p.\ $u_1 \leq u_0$, so in particular w.h.p.\ $u_1$ exists. A can determine $u_1$ but not $u_0$, as $u_0$ depends on the future.

A will follow a more sophisticated strategy from $u_1$ onward.  A will however play the remainder of the game on the graph $G_U$. By this we mean that she will ensure that if $\set{x,y}$ is a $G_U$-edge and $x$ has color $i$ at some stage, then she will not color $y$ with color $i$ even though this is strictly admissible.  

This weakens A and explains why our upper bound does not match our lower bound. On the other hand, if she can properly color $G_U$, then she will have succeeded in properly coloring $H=H_{n,p;3}$. B of course, does not play by these rules. We will show next that we can find a sequence $U=U_0\supseteq U_1\supseteq \cdots \supseteq U_\ell$ with the following properties: The $G_U$-edges of $U_i:(U_{i-1}\setminus U_i)$ between $U_i$ and $U_{i-1}\setminus U_i$ will be divided into two classes, {\em heavy} and {\em light}. Vertex $w$ is a heavy (resp. light) $G_U$-neighbor of vertex $v$ if the edge $(v,w)$ is $G_U$-heavy (resp. $G_U$-light).

\begin{enumerate}[{\bf (P1)}]
\item Each vertex of $U_i\setminus U_{i+1}$ has at most one light $G_U$-neighbor in $U_{i+1}$, for $0\leq i<\ell$.
\item All $U_i:(U_{i-1}\setminus U_i)$ $G_U$-edges are light for $i\geq 3$.
\item Each vertex of $U_i$ has at most $3\b/50$ $G_U$-heavy neighbors in $U_{i-1}\setminus U_i$ for $i = 1,2$.
\item Each vertex of $U_i\setminus U_{i+1}$ has at most $\b/3$ $G_U$-neighbors in $U_i$, for $0\leq i<\ell$.
\item $U_\ell$ contains at most one $G_U$-cycle.
\end{enumerate}

From this, we can deduce that the $G_U$-edges of $U_0$ can be divided up into the $G_U$-heavy edges $E_H$, $G_U$-light edges $F_L$, the $G_U$-edges inside $U_\ell$ and the rest of the $G_U$-edges. Assume first that $U_\ell$ does not contain a $G_U$-cycle. $\F=(U,F_L)$ is a forest and the strategy in \cite{R1}
can be applied. When attempting to color a vertex $v$ of $\F$, there are never more than three $\F$-neighbors of $v$ that have been colored. Since there are at most $\b/3+2 \cdot 3\b/50$ non-$\F$ neighbors, A will succeed since she has an initial list of size $\b/2$.

If $U_\ell$ contains a $G_U$-cycle $C$ then A can begin by coloring a vertex of $C$. This puts A one move behind in the tree coloring strategy, in which case we can bound the number of $\F$-neighbors by four.

It only remains to prove that the construction P1--P5 exists w.h.p.
Remember that $d$ is sufficiently large here.

We can assume without loss of generality that $|U_0|=2\g n$. This will not decrease the sizes of
the sets $a(v,U_0)$.

\subsection{The verification of P1--P4: Constructing $U_1$}\label{U1}
The general strategy will be as follows : We will consider two separate types of edge listed below. To tackle each type, we will formulate corresponding lemmas that will be presented subsequently. 
\begin{enumerate}[{\bf Type 1:}]
\item The edges $\{x,y\}$ in $G_U$ such that $\{x,y,z\} \not \in E$ for all $z \in U$. 
\item The remaining edges where for $\{x,y\}$ in $G_U$, there is $z \in U$ such that $\{x,y,z\} \in E$.
\end{enumerate}

Note that $d_U(v) \le d_{U,1}(v) + 2 d_{U,2}(v)$. Recall their definition just before Lemma \ref{lemma3}.

Let $L = 100$. Applying Lemmas \ref{lemma3} and \ref{lemma4} separately with
$$\s=2\g\text{ and }\th= \frac{ed^{\frac{1}{3} - \d} \log^2 d}{14}
\text{ and }\D=3\th+\frac{\b}L\text{ and }\t=\frac{\th}\b,$$
we see that w.h.p.
\[
S_1 = \set{v \in U_0 : d_{U_0,1}(v) \ge 3\th + \b/L} \text{ satisfies }
|S_1|\leq 2\t\g n=\frac{6e\log^3d}{d^{1 + 3\d}}n.
\]
\[
S_2 = \set{v \in U_0 : d_{U_0,2}(v) \ge 3\th + \b/L} \text{ satisfies }
|S_2|\leq 2\t\g n=\frac{6e\log^3d}{d^{1 + 3\d}}n.
\]
\[
U_{1,a}'=\set{v\in U_0:d_{U_0}(v)\geq 3\D}\text{ satisfies }
|U_{1,a}'|\leq 4\t\g n=\frac{12e\log^3d}{d^{1 + 3\d}}n.
\]

We then let $U_{1,a}\supseteq U_{1,a}'$ be the subset of $U_0$ consisting of the vertices with the $4\t\g n$
largest
values of $d_{U_0}$.

Let $A_0 = U_0 \setminus U_{1,a}$ and $B_0=\set{v\in U_{1,a}:|d_ {A_0}(v)| \ge \frac{3\b}{L}}$. 
Iteratively we define 
\beq{recurse1}{
A_i \coloneqq \set{v \in A_{i-1} :\; |d_{B_{i-1}}(v)| \ge 2}
}
\beq{recurse2}{
B_i \coloneqq \set{v \in B_{i-1} :\; |d_ {A_i}(v)| \ge \frac{3\b}{L}}
}

\begin{lemma} \label{lemmab'} 
W.h.p., $\not \exists$ disjoint sets $S, T\subseteq V(G), G=G_U$ such that 
$$t=|T| \leq t_0 = \frac{100\log d}{d^{\frac{2}{3} + \delta}}n\text{ and }|S| \geq s_0=\frac{8L|T|}{\beta} \text{ and } d_{T,1}(v) \geq \frac{\b}{L}\text{ for all }v \in S.$$ 
\end{lemma}

\begin{proof} 
We observe that if $S,T$ exist then one of the following two cases must occur:
\begin{enumerate}[{\bf C1:}]
\item 
$f(v) = |\set{u \in T : \exists w \in V(G)\setminus (S \cup T), \{u,w,v\} \in E(G)}| \geq \frac{\b}{2L}$ for at least $\frac{s_0}{2}$ vertices $v \in S$.
\item There are at least $t=\frac{\b s_0}{8L} $ hyperedges $\{u,v,w\}$ such that $u,v \in S$  and $w \in T$.   
\end{enumerate}

\begin{align*}
\mathbb{P}(\exists S, T \; \mbox{satisfying C1}) &\leq \sum_{t = \frac{\b}{L}}^{t_0} {\binom{n}{t}} {\binom{n}{\frac{8Lt}{\b}}}\; \left( {\binom{t}{\frac{\b}{2L}}} \bfrac{d}{n}^{\frac{\b}{2L}}\right)^{\frac{8Lt}{\b}}\\
&\leq \sum_{t = \frac{\b}{L}}^{t_0} \bfrac{ne}{t}^t \bfrac{ne\b}{8Lt}^{\frac{8Lt}{\b}} \bfrac{2Lted}{\b n}^{4t}  \\
&= \sum_{t = \frac{\b}{L}}^{t_0}\left(\bfrac{t}{n}^{3 - \frac{8L}{\b}} \frac{16 e^{5 + 8l/\b} \b^{8L / \b - 4} d^4}{L^{8l / \b - 4}}\right)^{t} \\
&\leq n\left(\bfrac{t_0}{n}^{3 - \frac{8L}{\b}} \frac{16 e^{5 + 8l/\b} \b^{8L / \b - 4} d^4}{L^{8l / \b - 4}}\right)^{\b/L} \\
&= o(1).
\end{align*}

\begin{align*}
\mathbb{P}(\exists S, T \; \mbox{satisfying C2})
&\leq \sum_{t = \frac{\b}{L}}^{t_0} \binom{n}{t} \binom{n}{\frac{8Lt}{\b}} \binom{\binom{\frac{8Lt}{\b}}{2}t}{t} \bfrac{d}{n^2}^t \\
&\leq \sum_{t = \frac{\b}{L}}^{t_0} \bfrac{ne}{t}^t \bfrac{ne\b}{8Lt}^{\frac{8Lt}{\b}} \bfrac{32 L^2 t^2 e d}{\b n^2}^{t}  \\
&\leq \sum_{t = \frac{\b}{L}}^{t_0} \left(\bfrac{t}{n}^{1- 8L/\b} \frac{32 e^{1+8L/\b} \b^{8L/\b -1} d}{L^{8L/\b - 2}}\right)^t\\
& \le \sum_{t=\frac{\b}L}^{\log^2n}\left(\bfrac{\log^2n}{n}^{1- 8L/\b} d^{1/3}\right)^t+\sum_{t=\log^2n}^{t_0}\left(\bfrac{t_0}{n}^{1- 8L/\b} d^{1/3}\right)^{\log^2n}\\
&= o(1).   
\end{align*}
\end{proof} 
Thus if $B' = \set{v \in U_{1,a} : d_{A_0,1}(v) \geq \frac{\b}{L}}$ then w.h.p.
$$|B'| \leq \frac{8L}{\b}|U_0| =  \frac{16L\g}{\b} n = \frac{2016L \; \log d}{d^{\frac{4}{3} +2\d}}n.$$

\begin{lemma}\label{lemmab''} W.h.p $\not\exists $ disjoint $S,T$ s.t 
\beq{xx1}{
t=|T| \leq t_0 = \frac{30n\; \log d}{d^{\frac{2}{3}+ \d}} and \; |S| \geq \frac{2L|T|}{\b}\text{ and }d_{T,2}(v) \geq \frac{\b}{L}\text{ for all }v \in S.
}
\end{lemma}
\begin{proof} 
\begin{align*}
\mathbb{P}(\exists S, T\text{ satisfying \eqref{xx1}}) 
&\leq  \sum_{t = \frac{\b}{L}}^{t_0}{\binom{n}{t}} {\binom{n}{\frac{2Lt}{\b}}}\left( {\binom{{\binom{t}{2}}} { \frac{\b}{L}}}\bfrac{d}{n^2}^{\frac{\b}{L}}\right)^{\frac{2Lt}{\b}}\\
&\leq \sum_{t = \frac{\b}{L}}^{t_0} \bfrac{ne}{t}^t \bfrac{ne\b}{2Lt}^{\frac{2Lt}{\b}} \left(\bfrac{t^2 dLe}{2\b n^2}^{\b/L} \right)^{\frac{2Lt}{\b}}\\
&=\sum_{t = \frac{\b}{L}}^{t_0} \bfrac{ne}{t}^t \bfrac{ne\b}{2Lt}^{\frac{2Lt}{\b}}\bfrac{t^2 dLe}{2\b n^2}^{2t} \\
&= \sum_{t = \frac{\b}{L}}^{t_0} \left(\bfrac{t}{n}^{3- 2L/\b} \frac{e^{3+ 2L/\b} \b^{2L/\b -2} d^2}{4 L^{2L/\b -2}} \right)^t  \\
&\leq \sum_{t=\frac{\b}L}^{\log^2n}\left(\bfrac{t}{n}^{3- 2L/\b} d^{2/3}\right)^{\b/L} + \sum_{t=\log^2n}^{t_0}\left(\bfrac{t_0}{n}^{3- 2L/\b} d^{2/3}\right)^{\log^2n}\\
&= o(1)
\end{align*}
\end{proof}

Thus if $B'' = \set{v \in U_{1,a} : d_{A_0,2}(v) \geq \frac{\b}{L}}$ then w.h.p. 
$$|B''| \le \frac{2}{\b}L|U_0| =  \frac{4L\g}{\b} n =\frac{504 L\log d}{ d^{\frac{4}{3} +2\d}}n.$$
Clearly, $B_0 \subseteq B' \cup B''$. Hence, 
$$|B_0| \leq |B' \cup B''| \le\frac{10L}{\b}|U_0|\leq \frac{3000L\log d}{ d^{\frac{4}{3} +2\d}}n.$$

Let $D_2(S) = \set {v : d_{S, 1}(v) \geq 2}$ for $S\subseteq V(G)$.
\begin{lemma} \label{lemma A'}
W.h.p., $|D_2(S)| < 3K|S|, \; \; \forall \; |S| \leq s_0=\frac{3000 L\log d}{d^{\frac{4}{3} + 2\delta}}n$ where $K = d^{\frac{2}{3}-2\delta}\log^2d$.
\end{lemma}

\begin{proof} 

Suppose that there exist $S$ and $T$ with $|S| \le s_0$ and $|T| = 3 \cdot K|S|$ such that $d_{s,1}(v) \ge 2$ for all $v \in T$. Then for $v \in T$, one of the following can occur.
\begin{enumerate}[{\bf D1:}]
\item There are $x,y \in S$ and $a,b \in T$ such that $\{v,x,a\},\{v,y,b\} \in E(G)$. 
\item There are $x,y \in S$ and $a \in T$, $b \in V(G) \setminus (S \cup T)$ such that $\{v,x,a\},\{v,y,b\} \in E(G)$.
\item There are $x,y \in S$ and $a,b \in V(G) \setminus (S \cup T)$ such that $\{v,x,a\},\{v,y,b\} \in E(G)$.
\end{enumerate} 

Now we construct $T'\subseteq T$ with $|T'| \ge K |S|$ such that if $v \in T$ then there exist $x,y \in S$ and $a,b \in V(G)\setminus (S \cup T')$ such that D1 holds. First, for every vertex of type D1, put $v$ in $T'$ and remove $a,b$ from further consideration. Second, for every vertex of type D2, put $v$ in $T'$ and remove $a$ from further consideration. Finally, for every vertex of type D3, put $v$ in $T'$. We observe that for every $v \in T'$ we have thrown away at most 2 vertices of $T$ and hence $|T'| \ge K |S|$. We will now estimate the probability of the existence of $S,T'$. 
\begin{align*}
\mathbb{P}(\exists \; |S| \leq s_0, |D_2(S)| \geq 3K|S|) 
&\leq \sum_{s=2}^{s_0} {\binom{n}{s}}{\binom{n}{Ks}}\left({\binom{s}{2}}n^2\left(\frac{d}{n^2}\right)^2\right)^{Ks}\\ 
&\leq \sum_{s=2}^{s_0} \bfrac{ne}{s}^s \bfrac{ne}{Ks}^{Ks} \left(\frac{s^2d^2}{2n^2}  \right)^{Ks}\\
&= \sum_{s=2}^{s_0} \left(\bfrac{s}{n}^{K-1} \frac{e^{K+1} d^{2K} }{K^K 2^K}\right)^s \\
&= o(1)
\end{align*}

\end{proof}

Thus if $A' = \set{v \in A_0 : v \in D_2(B_0)}$ then w.h.p. 
$$|A'| \leq \frac{9000\log^3d}{d^{\frac{2}{3}+4\delta}} n.$$
Let $D'_2(S) = \set {v : d_{S, 2}(v) \ge 1}$ for $S\subseteq V(G)$.
\begin{lemma}\label{lemma A''} 
W.h.p. $|D'_2(S)| \leq K|S|, \; \forall \; |S| \leq \frac{3000 L\log d}{d^{\frac{4}{3} + 2\delta}}n$ where $K = d^{\frac{2}{3}-2\delta}\log^2d$.
\end{lemma}

\begin{proof}
\begin{align*}
\mathbb{P}(\exists \; S \leq s_0, |D_2(S) \geq K|S|)
&\leq \sum_{s=2}^{s_0} {\binom{n}{s}}{\binom{n}{Ks}}\left( {\binom{s}{2}} \frac{d}{n^2}\right)^{Ks}\\
&\leq \sum_{s=2}^{s_0} \bfrac{ne}{s}^s \bfrac{ne}{Ks}^{Ks} \left(\frac{s^2 d}{2n^2} \right)^{Ks}   \\
&= \sum_{s=2}^{s_0} \left(\bfrac{s}{n}^{K-1} \frac{e^{K+1}d^K }{2^K K^K} \right)^s\\
&= o(1).
\end{align*}
\end{proof}
So w.h.p. $|A_1| = |A' \cup A''| \leq \frac{12000\log^3d}{d^{\frac{2}{3}+4\delta}}n$. From \eqref{recurse1}, \eqref{recurse2}, Lemma \ref{lemmab'} and Lemma \ref{lemmab''} we see that  
\begin{align*}
|B_1|& \leq \frac{10L}{\b}|A_1|.\\
|A_2|&\leq 4K |B_1|.\\
|B_i|& \leq \frac{10L}{\beta} |A_i|\leq \frac{40KL}{\beta}|B_{i-1}|\leq \bfrac{40KL}{\beta}^i|B_0|.
\end{align*}
Using Lemmas \ref{lemma A'} and Lemma \ref{lemma A''},
$$|A_{i+1}|\leq 4K |B_i| \leq 4K \left(\frac{40KL}{\beta} \right)^i |B_0| \leq  4K \left(\frac{40KL}{\beta}\right)^i \cdot \frac{3000L\log d}{d^{\frac{4}{3}+ 2\delta}}n,$$
where $\frac{KL}{\beta} = 150d^{-3\delta}\log^2d$.\\

Let $\z= \rdup{\frac{2}{\d}}$ and let $Y = N(B_{\z}) \cap A_{\z}$. Then,
\begin{align*}
|Y|\le |A_{\z}|&\le 4K\left(\frac{40KL}{\b}\right)^{\z -1} \cdot \frac{3000\log d}{d^{4/3 + 2\d}}n \\
&= 12000(120L)^{\z-1} \cdot \frac{(\log d)^{1+2\z}}{d^{2/3 + \d + \d \z}}n\\
&\le \frac{n}{d^{\z\delta}}\leq \t\g n.
\end{align*}
Let $U_1 = U_{1,a} \cup Y$. Then 
$$|U_1| \leq  5\t\g n=\frac{15 e \log^3 d}{d^{1 + 3 \d}}.$$ 
We now define the light and heavy edges in the following fashion, 
\begin{enumerate}[{\bf Q1:}]
\item The edges between $B_i$ and $ A_i \setminus A_{i+1}$ are light
\item The edges between $B_i \setminus B_{i+1}$ and $A_{i+1}$ are heavy 
\item The edges between $U_1 \setminus U_{1,a}$ and $U_0 \setminus U_1$ are heavy
\end{enumerate}

We now check that \textbf{P1-P4} hold. First consider the light edges. For every vertex $v \in U_0 \setminus U_1$ there is at most one light neighbour in $U_1$. Because if $v \in A_i$ and $v \not \in A_{i+1}$ and there are 2 light neighbors $x,w$ of $v$ in $U_1$, by Q1, $x,w \in B_i$ and that would contradict the fact that $v \not \in A_{i+1}$. This implies that \textbf{P1} holds. 

We will argue next that for all $v \in U_1$, $d_{U_0 \setminus U_1} (v) \le 3 \D \le \frac{3\b}{50}$. For $v \not \in U_{1,a}$ this is true from the definition of $U_{1,a}$. Similarly, for $v\notin B_0$. Now consider $v \in B_i \setminus B_{i+1},i\geq 0$. It only has light neighbors in $ A_i \setminus A_{i+1}$ and if $v$ has more than $\frac{3\b}{L}$ heavy neighbors in $A_{i+1}$ then $v$ should be in $B_{i+1}$, which is a contradiction. Because it is also in $B_j,j\leq i-1$ it only has light neighbors in $(A_{i-1}\setminus A_i)\cup (A_{i-2}\setminus A_{i-3})\cup \cdots=A_0\setminus A_i$. Clearly \textbf{P3, P4} hold.

\subsection{The verification of P1--P4: Constructing $U_2$}
Applying Lemma \ref{lemma3} and \ref{lemma4} separately with
$$\s=\frac{15 e \log^3 d}{d^{1 + 3 \d}} \text{ and }\th= \frac{L}{\d} \text{ and }\D=3\th+\frac{\b}L\text{ and }\t=\frac{\th}\b,$$
we see that w.h.p.
\[
S_1 = \set{v \in U_1 : d_{U_1,1}(v) \ge 3\th + \b/L} \text{ satisfies }
|S_1|\leq \t \s n=\frac{45 eL\log^3d}{\d d^{5/3 + 4\d}}n.
\]
\[
S_2 = \set{v \in U_1 : d_{U_1,2}(v) \ge 3\th + \b/L} \text{ satisfies }
|S_2|\leq \t \s n=\frac{45 e L \log^3d}{\d d^{5/3 + 4\d}}n.
\]

\[
U_{2,a}'=\set{v\in U_1:d_{U_1}(v)\geq 3\D}\text{ satisfies }
|U_{2,a}'|\leq 2 \t\s n=\frac{90 e L \log^3d}{\d d^{5/3 + 4\d}}n.
\]

We then let $U_{2,a}\supseteq U_{2,a}'$ be the subset of $U_1$ consisting of the vertices with the $2\t\s n$ largest values of $d_{U_1}$. As in Section \ref{U1}, define $A_0 = U_1 \setminus U_{2,a}$ and let $B_0=\set{v\in U_{2,a}:|d_ {A_0}(v)| \ge \frac{3\b}{L}}$. Iteratively we define 
\[
A_i \coloneqq \set{v \in A_{i-1} :\; |d_{B_{i-1}}(v)| \ge 2}
\]
\[
B_i \coloneqq \set{v \in B_{i-1} :\; |d_ {A_i}(v)| \ge \frac{3\b}{L}}
\]

Let $B' = \set{v \in U_{2,a} : d_{A_0,1}(v) \geq \frac{\b}{L}}$. Using Lemma \ref{lemmab'}, we see that w.h.p., $|B'| \leq \frac{8L|U_1|}{\b} \leq \frac{360 e L \; \log^3 d}{d^{\frac{5}{3} + 4\d}} n$.

Let $B'' = \set{v \in U_{2,a} : d_{A_1,2}(v) \geq \frac{\b}{L}}$. Using Lemma \ref{lemmab''}, we see that w.h.p. $|B''| \le \frac{90 e L \; \log^3 d}{d^{\frac{5}{3} + 4\d}} n$.\\
Clearly, $B_0 \subseteq B' \cup B''$. Therefore, w.h.p., 
$$|B_0| \leq |B' \cup B''| \le \frac{450 e L \; \log^3 d}{d^{\frac{5}{3} + 4\d}} n.$$

Arguing as in Section \ref{U1} we see that w.h.p. 
\begin{align*}
B_i|& \leq \frac{10L}{\beta} |A_i|\leq \frac{40KL}{\beta}|B_{i-1}|\leq \bfrac{40KL}{\beta}^i|B_0|.\\
|A_{i+1}|&\leq 4K |B_i| \leq 4K \left(\frac{40KL}{\beta} \right)^i |B_0|
\end{align*}

Remember that $\z= \rdup{\frac{2}{\d}}$ and let $Y = N(B_{\z}) \cap A_{\z}$. Then,
\begin{align*}
|Y|\le |A_{\z}|&\le 4K\left(\frac{40KL}{\b}\right)^{\z -1} \cdot \frac{3000\log d}{d^{4/3 + 2\d}}n \leq \frac{n}{d^2}.\\
\end{align*}
Let $U_1 = U_{1,a} \cup Y$. Then w.h.p.
$$|U_1| \leq  5\t\g n=\frac{15 e \log^3 d}{d^{1 + 3 \d}}.$$ 
We now define the light and heavy edges in the following fashion, 
\begin{enumerate}[{\bf Q1:}]
\item The edges between $B_i$ and $ A_i \setminus A_{i+1}$ are light
\item The edges between $B_i \setminus B_{i+1}$ and $A_{i+1}$ are heavy 
\item The edges between $U_1 \setminus U_{1,a}$ and $U_0 \setminus U_1$ are heavy
\end{enumerate}

We now check that \textbf{P1-P4} hold. First consider the light edges. For every vertex $v \in U_0 \setminus U_1$ there is at most one light neighbour in $U_1$. Because if $v \in A_i$ and $v \not \in A_{i+1}$ and there are 2 light neighbors $x,w$ of $v$ in $U_1$, by Q1, $x,w \in B_i$ and that would contradict the fact that $v \not \in A_{i+1}$. This implies that \textbf{P1} holds. 

We will argue next that for all $v \in U_1$, $d_{U_0 \setminus U_1} (v) \le 3 \D \le \frac{3\b}{50}$. For $v \not \in U_{1,a}$ this is true from the definition of $U_{1,a}$. Similarly, for $v\notin B_0$. Now consider $v \in B_i \setminus B_{i+1},i\geq 0$. It only has light neighbors in $ A_i \setminus A_{i+1}$ and if $v$ has more than $\frac{3\b}{L}$ heavy neighbors in $A_{i+1}$ then $v$ should be in $B_{i+1}$, which is a contradiction. Because it is also in $B_j,j\leq i-1$ it only has light neighbors in $(A_{i-1}\setminus A_i)\cup (A_{i-2}\setminus A_{i-3})\cup \cdots=A_0\setminus A_i$. Clearly \textbf{P3, P4} hold.

\subsection{The verification of P1--P4: Constructing $U_2$}
Applying Lemma \ref{lemma3} and \ref{lemma4} separately with
$$\s=\frac{15 e \log^3 d}{d^{1 + 3 \d}} \text{ and }\th= \frac{L}{\d} \text{ and }\D=3\th+\frac{\b}L\text{ and }\t=\frac{\th}\b,$$
we see that w.h.p.
\[
S_1 = \set{v \in U_1 : d_{U_1,1}(v) \ge 3\th + \b/L} \text{ satisfies }
|S_1|\leq \t \s n=\frac{45 eL\log^3d}{\d d^{5/3 + 4\d}}n.
\]

\[
S_2 = \set{v \in U_1 : d_{U_1,2}(v) \ge 3\th + \b/L} \text{ satisfies }
|S_2|\leq \t \s n=\frac{45 e L \log^3d}{\d d^{5/3 + 4\d}}n.
\]

\[
U_{2,a}'=\set{v\in U_1:d_{U_1}(v)\geq 3\D}\text{ satisfies }
|U_{2,a}'|\leq 2 \t\s n=\frac{90 e L \log^3d}{\d d^{5/3 + 4\d}}n.
\]

We then let $U_{2,a}\supseteq U_{2,a}'$ be the subset of $U_1$ consisting of the vertices with the $2\t\s n$ largest values of $d_{U_1}$. As in Section \ref{U1}, define $A_0 = U_1 \setminus U_{2,a}$ and let $B_0=\set{v\in U_{2,a}:|d_ {A_0}(v)| \ge \frac{3\b}{L}}$. Iteratively we define 
\beq{recurse1b}{
A_i \coloneqq \set{v \in A_{i-1} :\; |d_{B_{i-1}}(v)| \ge 2}
}
\beq{recurse2b}{
B_i \coloneqq \set{v \in B_{i-1} :\; |d_ {A_i}(v)| \ge \frac{3\b}{L}}
}

Let $B' = \set{v \in U_{2,a} : d_{A_0,1}(v) \geq \frac{\b}{L}}$. Using Lemma \ref{lemmab'}, we see that w.h.p., $|B'| \leq \frac{8L|U_1|}{\b} \leq \frac{360 e L \; \log^3 d}{d^{\frac{5}{3} + 4\d}} n$.

Let $B'' = \set{v \in U_{2,a} : d_{A_1,2}(v) \geq \frac{\b}{L}}$. Using Lemma \ref{lemmab''}, we see that w.h.p. $|B''| \le \frac{90 e L \; \log^3 d}{d^{\frac{5}{3} + 4\d}} n$.\\
Clearly, $B_0 \subseteq B' \cup B''$. Therefore, w.h.p., 
$$|B_0| \leq |B' \cup B''| \le \frac{450 e L \; \log^3 d}{d^{\frac{5}{3} + 4\d}} n.$$

Arguing as in Section \ref{U1} we see that w.h.p. 
\begin{align*}
B_i|& \leq \frac{10L}{\beta} |A_i|\leq \frac{40KL}{\beta}|B_{i-1}|\leq \bfrac{40KL}{\beta}^i|B_0|.\\
|A_{i+1}|&\leq 4K |B_i| \leq 4K \left(\frac{40KL}{\beta} \right)^i |B_0|
\end{align*}
With $\z = \rdup{\frac{2}{\d}}$ and $K=d^{\frac{2}{3}-2\delta}\log^2d$ as before and $Y = N(B_{\z}) \cap A_{\z}$ we get 
\begin{align*}
|Y|\le |A_{\z}|&\le 4K\left(\frac{40KL}{\b}\right)^{\z -1} \cdot \frac{450 e L \; \log^3 d}{d^{5/3 + 4\d}} n\leq \frac{n}{d^2}.\\
\end{align*}
Letting $\g_2 = \frac{500Le\log^3d}{\d d^{5/3 + 4\d}}$ and $U_2 = U_{2,a} \cup Y$ we see that w.h.p. $|U_2| \le \g_2 n$.

We can define heavy and light edges as in $U_1$ and \textbf{P1-P4} follows. 

\subsection{The verification of P1--P4: Constructing $U_3$}\label{U3}

Applying Lemma \ref{lemma3} and \ref{lemma4} separately with
$$\s=\frac{500eL\log^3d}{\d d^{5/3 + 4\d}} \text{ and }\th= \frac{5}{2} \text{ and }\D=3\th+\frac{\b}{L}\text{ and }\t=\frac{L\th}{\b},$$
we see that w.h.p.
\[
S_1 = \set{v \in U_2 : d_{U_2,1}(v) \ge 3\th + \b/L} \text{ satisfies }
|S_1|\leq \t \s n=\frac{1250L e \log^3d}{\d d^{7/3 + 5\d}}n.
\]

\[
S_2 = \set{v \in U_2 : d_{U_2,2}(v) \ge 3\th + \b/L} \text{ satisfies }
|S_2|\leq \t \s n=\frac{1250L e \log^3d}{\d d^{7/3 + 5\d}}n.
\]

\[
U_{3}'=\set{v\in U_2:d_{U_2}(v)\geq 3\D}\text{ satisfies }
|U_{3}'|\leq 2 \t\s n=\frac{2500L e  \log^3d}{\d d^{7/3 + 5\d}}n.
\]

We now construct $U_3 \supseteq U_3'$ by repeatedly adding vertices $y_1, y_2,..y_s$ of $U_2 \setminus U_3' $ such that $y_j$ is the lowest numbered vertex not in $ Y_j =  U_3' \cup \set{y_1, y_2....y_{j-1}}$ that has at least two neighbors in $Y_j$ in $G$. W.h.p., this process terminates with $j=t \le 39 |U_3'|$. We can apply Lemma \ref{lemma2} to see that w.h.p. this does not happen. Indeed, let $S_0 = U_3'$. We add vertices to $S_{j-1}$ to create the set $S_j$ iteratively. In this procedure we encounter two cases.
\begin{itemize}
\item $\exists x,w \in Y_j$ such that $(y_j, x, w) \in E(G)$. Then $S_j = S_{j-1} \cup \{y_j\}$. 
\item $\exists x,w \in Y_j \; \mbox{and} \; a,b \not \in Y_j$ such that $(x,a,y_j),(w,b,y_j) \in E(G)$. Then $S_j = S_{j-1} \cup \set{y_j, a, b}$
\end{itemize}

Note that we are adding at least 2 hyper-edges for every 3 vertices added to $S_t$. If $s \ge 39 |U_3'|$, then $e_3(S_s) \ge \frac{13}{20} |S_s|$. Apply Lemma \ref{lemma2} with $\theta = \frac{13}{20}$ and $\s = 120 \cdot \frac{2500L e \log^3d}{\d d^{7/3 + 5\d}}$ to conclude that $t \le 39 |U_3'|$.

Putting $U_3=U_3'\cup S_t$ we see that each vertex in $U_2\setminus U_3$ has at most one $G$-neighbor in $U_3$. We can therefore make the $U_3:(U_2\setminus U_3)$ edges light and satisfy P1, P2 and P4.
\subsection{The verification of P1-P5 : Construction of $U_i, i \ge 4$}

We repeat the argument of Section \ref{U3} to construct the rest of the sequence $U_0 \supseteq U_1 \supseteq U_2 \supseteq ... \supseteq U_l$. One can check that $|U_i| \le \frac{200L}{\b} |U_{i-1}|$. We choose $l$ so that $|U_l| \le \log n$. We can then easily prove that w.h.p. $S$ contains at most $|S|$ edges of $G$ whenever $|S| \le \log n$, implying P5.
\section{Final remarks}
We have shown lower bounds for the game chromatic number of random $k$-uniform hypergraphs and upper bounds for random 3-uniform hypergraphs. The lower bound is satisfactory in that it is within a constant factor of the chromatic number. The upper bound is most likely not tight, but it is still non-trivial in that it is much smaller than $d$. 

We conjecture that the upper bound can be reduced to within a constant factor of the lower bound. It would also be of interest to consider upper bounds for $k$-uniform hypergraphs, $k\geq 4$.

\end{document}